\documentclass[12pt,oneside, usenames,dvipsnames]{amsart}

\usepackage[utf8]{inputenc}
\usepackage[margin=0.9in]{geometry}
\usepackage{amssymb,amsthm,color,hyperref}
\usepackage{array}
\usepackage{tikz}
\usetikzlibrary{cd}
\usetikzlibrary{decorations.markings}
\usepackage{pgfplots}
  \pgfplotsset{compat=1.17}
\usepackage{mathtools}
\usepackage{xcolor}

  \newtheorem{theorem}{Theorem}
  \newtheorem{lemma}[theorem]{Lemma}

\renewcommand{\le}{\leqslant}

\newcommand{\psin}{\mathsf{PSI}}
\newcommand{\copsi}{\mathsf{CoPSI}}

\title{What is the maximal connected partial symmetry index of a connected graph of a given size?}

\author{Z. Janelidze, F. van Niekerk and J. Viljoen}

\address{Department of Mathematical Sciences\\ Stellenbosch University, South Africa\\
and\\
National Institute for Theoretical and Computational Sciences (NITheCS), South Africa}
\email{zurab@sun.ac.za}

\address{Department of Mathematical Sciences\\ Stellenbosch University, South Africa\\
and\\
National Institute for Theoretical and Computational Sciences (NITheCS), South Africa}
\email{fkvn@sun.ac.za}

\address{Reddam House Sydney, 68/70 Edgecliff Road, Woollahra, NSW, 2025, Australia}
\email{jade.villz@gmail.com}

\begin{document}

\maketitle

\section*{Abstract} For a given graph, by its \emph{connected partial symmetry index} we mean the number of all isomorphisms between connected induced subgraphs of the graph. In this brief note we answer the question in the title.

\section*{Introduction}

Consider a graph $G=(G,E)$, i.e., a set $G$ equipped with a set $E$ of two-element subsets of $G$. By a \emph{partial symmetry} in $G$ we mean a triple $(U,V,f)$, where $U$ and $V$ are induced subgraphs of $G$ and $f$ is an isomorphism $f\colon X\to Y$. When one of $U$ or $V$ is connected, then both are and we call $(U,V,f)$ a \emph{connected partial symmetry}. We are interested in the number of connected partial symmetries in a graph: call it the \emph{connected partial symmetry index} of the graph, $\copsi(G)$. For a graph with $n$ vertices (i.e., $|G|=n$), the largest attainable connected partial symmetry index is clearly when $E$ is largest, that is, when the graph is complete, $G\approx\mathsf{K}_n$. This is because in $\mathsf{K}_n$, every non-empty partial symmetry is a connected partial symmetry. Moreover, the largest connected partial symmetry index is attained only by the complete graphs for a given order $n$ of the graph: if there is at least one pair $x\neq y$ of vertices in a graph that are not connected by an edge, then there will be no connected partial symmetry $(U,V,f)$ with $U=\{x,y\}$ (since such $U$ is not connected). In this paper we show that a star with $n$ rays (the bipartite graph $\mathsf{K}_{1,n}$) plays a similar role as the complete graph if instead of the order of the graph we restrict graph's size (i.e., the number of edges), and consider only connected graphs. Note that if we allow disconnected graphs, then the connected partial symmetry index is unbounded since we can always create new connected partial symmetries without increasing the size of the graph, by adding vertices to the graph of degree $0$.

\begin{theorem}\label{ThmC}
  Let $G$ be a connected graph of size $n$. We have
  \[\copsi(G)\le\copsi(\mathsf{K}_{1,n})\]
  and furthermore, if $\copsi(G)=\copsi(\mathsf{K}_{1,n})$ then $G$ is isomorphic to $\mathsf{K}_{1,n}$.
\end{theorem}

The question in the title is then answered by establishing a formula for the connected partial symmetry index of a star with $n$ rays, which we find to be:
$$\copsi(\mathsf{K}_{1,n})=2n(n+1)+\sum_{i=0}^{n}\binom{n}{i}^2i!$$
Interestingly, the sequence $\copsi(\mathsf{K}_{1,n})$ does not appear in the On-Line Encyclopedia of Integer Sequences (OEIS). Below are illustrations of $\mathsf{K}_{1,n}$ when $n=1,3,9$.
$$ \scalebox{0.8}
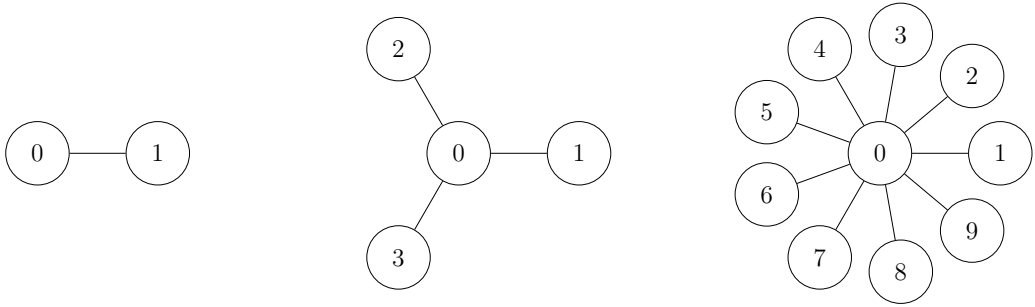
{\begin{tikzpicture}
  \node[draw, circle, minimum size=30pt]
  (0) at (0,0){$0$};
  \foreach \i in {1,...,1} {
    \node[draw, circle, minimum size=30pt] (\i) at ({2*cos(360/1*(\i-1))},{2*sin(360/1*(\i-1))}){$\i$};
    \draw (0) -- (\i);
  }
  \node[draw, circle, minimum size=30pt]
  (0) at (7,0){$0$};
  \foreach \i in {1,...,3} {
    \node[draw, circle, minimum size=30pt] (\i) at ({7+2*cos(360/3*(\i-1))},{2*sin(360/3*(\i-1))}){$\i$};
    \draw (0) -- (\i);
  }
    \node[draw, circle, minimum size=30pt]
  (0) at (14,0){$0$};
  \foreach \i in {1,...,9} {
    \node[draw, circle, minimum size=30pt] (\i) at ({14+2*cos(360/9*(\i-1))},{2*sin(360/9*(\i-1))}){$\i$};
    \draw (0) -- (\i);
  }
\end{tikzpicture}}$$

\section*{Proof of Theorem~\ref{ThmC}}

%Connected partial symmetry index of a complete graph $\mathsf{K}_n$ is clearly the same as the number of partial symmetries of $n$-element set, which is well known to be
%$$\copsi(\mathsf{K}_n)=\sum_{i=0}^{n}\binom{n}{i}^2i!$$
There are three types of connected partial symmetries $(U,V,f)$:
  \begin{itemize}
  \item \emph{singleton symmetries}, where $|U|=|V|=1$,
  \item \emph{edge symmetries}, where $U$ and $V$ are edges,
  \item the remaining connected partial symmetries, where the subgraphs induced by $U$ and $V$ each have size greater or equal to $2$.
  \end{itemize}
  For any graph $G$, the number of singleton symmetries is always $m^2$, where $m$ is the order of the graph. $\mathsf{K}_{1,n}$ has order $n+1$. A graph of size $n$ cannot have higher order (the size of a connected graph of order $m$ is at least $m-1$). So a graph whose size is $n$ always has less or the same number of singleton symmetries as $\mathsf{K}_{1,n}$ does. The same is true for edge symmetries. In fact, any graph of size $n$ has exactly $2n^2$ many edge symmetries: each edge can be partial symmetryed to any other partial symmetry in exactly two ways.

  To analyze partial symmetries of the third kind, we need the following.

  \begin{lemma}\label{LemA}
If two isomorphisms $u$ and $v$ from a graph $G$ to a graph $H$ induce the same bijection of edge sets, then $u=v$ provided $G$ is connected and has size different from $1$.     
\end{lemma}

\begin{proof} Connected graphs of size $0$ are singletons, so the result holds trivially in this case. Consider a connected graph $G$ of size $n>1$. Let $u,v$ be two isomorphisms $G\to H$ that determine the same bijection between the edges sets of $G$ and $H$. Suppose $u(a)\neq v(a)$ for some vertex $a$. Since the order of $G$ is greater than $1$ and $G$ is connected, there is at least one edge $\{a,b\}$. Since $\{u(a),u(b)\}=\{v(a),v(b)\}$ and $u(a)\neq v(a)$, we must have that $u(a)=v(b)$ and $u(b)=v(a)$. Let $c$ be any other vertex in $G$ (it exists since $n>1$). $G$ is connected, so there is a path connecting $c$ with $a$. This implies that either there is an edge $\{a,d\}$ or an edge $\{b,d\}$ where $d\neq a$ and $d\neq b$. Suppose there is such edge $\{a,d\}$. Since $u(a)\neq v(a)$ and $u=v$ on edges, we must have $u(d)=v(a)=u(b)$. Since $u$ is an isomorphism, this is a contradiction. Existence of an edge $\{b,d\}$ leads to a similar contradiction.
\end{proof}

  By Lemma~\ref{LemA}, for a given connected partial symmetry $(U,V,f)$ in any graph, $f$ is uniquely determined by where $f$ maps the edges of $U$. When $U$ has at least one edge, it is itself uniquely determined by its edge set. So every partial symmetry $(U,V,f)$ of the third type is uniquely determined by a non-singleton permutation of the edge set. Now, in the star $\mathsf{K}_{1,n}$, any non-singleton permutation of the edge set defines a connected partial symmetry of the third type. Together with the observations we made about singleton and edge symmetries before the lemma, this shows that a connected graph $G$ of size $n$ cannot have more connected partial symmetries than $\mathsf{K}_{1,n}$. 

  Suppose now $G$ is a connected graph of size $n$ such that $\copsi(G)=\copsi(\mathsf{K}_{1,n})$. Then $G$ cannot have order less than $n+1$, since otherwise it would lose out on some of the singleton symmetries. It cannot have order greater than $n+1$ either, since then it would have larger size: it is a well known result in graph theory that a connected graph of order $m$ has at least $m-1$ many edges. So $G$ has exactly $n+1$ vertices. This makes $G$ a tree, by another well known result that a connected graph of size $n$ and order $n+1$ does not have cycles.
  
  Now, if every vertex in $G$ has degree less than $3$, then $G$ is a path. If $n\leqslant 2$ then $G\approx \mathsf{K}_{1,n}$. If $n>2$ then $G$ contains as an induced subgraph a path $U$ of length $3$. This means that $G$ would loose out on partial symmetries of the third type, since not every permutation of the edge set of $U$ will be a partial symmetry (if the middle edge in the path is fixed, the outer edges cannot be swapped). So if every vertex has degree less than $3$, then $G\approx \mathsf{K}_{1,n}$. 
  
  Consider now the case when there is at least one vertex whose degree is $3$ or more. Consider all edges on which this vertex lies. If the other vertex on each edge has degree $1$, then, since $G$ is connected, it cannot have any other vertices and hence $G\approx \mathsf{K}_{1,n}$. Suppose there is at least one edge where the other vertex has degree greater than $1$. Then, since $G$ is a tree, it will contain as an induced subgraph a path $U$ of length $3$. As before, in this case $G$ would loose out on partial symmetries of the third type. So in this case too, $G\approx \mathsf{K}_{1,n}$.

  The proof of Theorem~\ref{ThmC} is now complete.

\section*{Additional Remarks}

From the proof of Theorem~\ref{ThmC}, we can extract a formula for $\copsi(\mathsf{K}_{1,n})$, by counting partial symmetries of each of the types considered there:
\begin{align*}
\copsi(\mathsf{K}_{1,n}) &=(n+1)^2+2n^2+\sum_{i=2}^n \binom{n}{i}^2 i!\\
&=2n^2+2n+\binom{n}{0}^2 0!+\binom{n}{1}^2 1!+\sum_{i=2}^n \binom{n}{i}^2 i!\\
&=2n(n+1)+\sum_{i=0}^n \binom{n}{i}^2 i!
\end{align*}
The second summand in the last line is in fact the number of partial symmetries of a connected graph of order $n$, i.e., the \emph{partial symmetry index} of $\mathsf{K}_n$. So
$$\copsi(\mathsf{K}_{1,n})=2n(n+1)+\psin(\mathsf{K}_{n}).$$
While the sequence $\copsi(\mathsf{K}_{1,n})$ itself does not appear in OEIS, the two terms above do: the first one is A046092 and the second one is A002720. 

Note that we could also rewrite the formula above as
$$\copsi(\mathsf{K}_{1,n})=n^2+(n+1)^2+\copsi(\mathsf{K}_{n}),$$
since every partial symmetry of $\mathsf{K}_{n}$ except the empty one is connected.

A path $\mathsf{P}_{n}$ of order $n\geqslant 1$ has only two automorphisms (identity and flip). This makes it easy to obtain the connected partial symmetry index of $\mathsf{P}_{n}$:
\begin{align*}
\copsi(\mathsf{P}_n)
&=n^2+\sum_{i=1}^{n-1} 2(n-i)^2\\
&=n^2+2\sum_{i=1}^{n-1} i^2\\
&=n^2+2\frac{n(n-1)(2n-1)}{6}\\
&=\frac{3n^2+n(n-1)(2n-1)}{3}\\
&=\frac{(3n+(n-1)(2n-1))n}{3}\\
&=\frac{(2n^2+1)n}{3}
\end{align*}
These numbers are best known as the octahedral numbers: the number of spheres in an octahedron formed from close-packed spheres. The sequence appears in OEIS as A212133, but our interpretation does not show up there. We can also easily establish the connected partial symmetry index of a cycle of order $n$, where $n\geqslant 3$:
\begin{align*}
\copsi(\mathsf{C}_n) &= n^2+2(n-2)n^2+2n\\
&=n(2n^2-3n+2)
\end{align*}
This sequence is much less known, although it still appears in OEIS, as A212133. Once again, our interpretation is not given. This is an indication that the present paper may be the first one on connected partial symmetry indices.  

In fact, there is not much literature on partial symmetry indices of particular graphs either, despite of the fact that there is literature on partial symmetries of graphs (see, e.g., \cite{ChiPle17, JJSS20, JL}).

\section{Acknowledgement}

This paper grew out from the Honours Project of the third author at Stellenbosch University, completed under the supervision of the first author in 2020.

\end{document}